\documentclass[12pt,reqno,oneside]{amsart}
\usepackage[headheight=15pt,rmargin=0.5in,lmargin=0.5in,tmargin=0.75in,bmargin=0.75in]{geometry}
\usepackage{imakeidx}
\usepackage{framed}
\usepackage{amsthm}
\usepackage{amssymb}
\usepackage{amsmath}
\usepackage{mathrsfs}
\usepackage{enumitem}
\usepackage{hyperref}
\usepackage{array,makecell}
\usepackage[capitalise,noabbrev]{cleveref}
\usepackage{appendix}
\usepackage{tikz}
\usepackage{tikz-cd}
\usepackage{nomencl}\makenomenclature
\usetikzlibrary{braids,arrows,decorations.markings}

\newtheorem{theorem}{Theorem}[section]
\newtheorem{proposition}{Proposition}[section]

\newtheorem{lemma}{Lemma}[section]

\crefname{method}{Method}{Methods}

\newenvironment{stabular}[2][1]
  {\def\arraystretch{#1}\tabular{#2}}
  {\endtabular}

\newcommand{\legendre}[2]{\left(\frac{#1}{#2}\right)}

\newcommand{\tlegendre}[2]{\textstyle{\left(\frac{#1}{#2}\right)}}
\newcommand{\tmod}[1]{\ \left(\text{mod }#1\right)}

\newcommand{\mc}{\mathcal}

\newcommand{\GL}{\mathrm{GL}}

\newcommand{\Z}{\mathbb{Z}}

\newcommand{\R}{\mathbb{R}}
\newcommand{\C}{\mathbb{C}}

\renewcommand{\H}{\mathbb{H}}

\renewcommand{\a}{\alpha}

\newcommand{\g}{\gamma}
\renewcommand{\d}{\delta}
\newcommand{\z}{\zeta}
\renewcommand{\t}{\theta}

\renewcommand{\l}{\lambda}
\newcommand{\s}{\sigma}
\newcommand{\w}{\omega}

\newcommand{\G}{\Gamma}

\renewcommand{\L}{\Lambda}

\newcommand{\e}{\varepsilon}

\newcommand{\x}{\times}

\newcommand{\lf}{\lfloor}
\newcommand{\rf}{\rfloor}
\newcommand{\wtilde}{\widetilde}

\newcommand{\conj}{\overline}

\DeclareMathOperator{\ord}{\mathrm{ord}}

\DeclareMathOperator*{\Res}{\mathrm{Res}}

\renewcommand{\Re}{\mathrm{Re}}

\DeclareMathOperator{\cc}{\mathbb{C}}
\DeclareMathOperator{\zz}{\mathbb{Z}}
\DeclareMathOperator{\rr}{\mathbb{R}}
\DeclareMathOperator{\qq}{\mathbb{Q}}
\DeclareMathOperator{\inv}{^{-1}}

\newenvironment{psmallmatrix}
  {\left(\begin{smallmatrix}}
  {\end{smallmatrix}\right)}

\title{A converse theorem in half-integral weight}
\author{Steven Creech}
\address{Department of Mathematics \\ Brown University \\ 167 Thayer St, Room 110 \\ Providence, RI 02906 \\ USA}
\email{steven\_creech@brown.edu}

\author{Henry Twiss}
\address{Department of Mathematics \\ Brown University \\ 167 Thayer St, Room 110 \\ Providence, RI 02906 \\ USA}
\email{henry\_twiss@brown.edu}

\thanks{\textbf{Acknowledgements}: The authors would like to thank Min Lee for suggesting the problem and for helpful conversations. The authors would also like to thank Jeff Hoffstein and Junehyuk Jung for additional conversations. The second author was supported by the NSF GRFP.}

\date{\today}

\makeindex

\begin{document}

\begin{abstract}
    In this paper, we prove a converse theorem for half-integral weight holomorphic forms assuming functional equations for $L$-series with additive twists. This result is an extension of Booker, Farmer, and Lee's result in \cite{booker2022extension} to the half-integral weight setting. Similar to their work, the main result of this paper is obtained as a consequence of the half-integral weight Petersson trace formula.
\end{abstract}

\maketitle

\section{Introduction}
  A part of the Langlands program seeks to study automorphic forms on GL(n) and other groups via $L$-functions. Two of the main tools used in this setting are trace formulas and converse theorems. Converse theorems provide sufficient conditions to classify Dirichlet series possessing sufficiently many functional equations as automorphic $L$-functions. The first converse theorem was due to Hamburger (see \cite{hamburger1921riemannsche}) in 1921 which uniquely classified the zeta function $\z(s)$ by its functional equation. Just over 15 years later (see \cite{hecke1936bestimmung}), Hecke extended the work of Hamburger by classifying $L$-functions of cusp forms on the full modular group by their functional equation. Some years later, Weil further extend Hecke’s converse theorem to cusp forms of higher level. For Weil’s result, one needs to assume not only the functional equation of the underlying $L$-function but also the functional equations for all of its twists by quadratic Dirichlet characters (see \cite{weil1967bestimmung}). In 2002, Venkatesh’s was able to prove Hecke’s result using additively twisted $L$-series via Voronoi summation (see \cite{venkatesh2002limiting}) instead of twists by quadratic Dirichlet characters. Recently, Booker, Farmer, and Lee (in \cite{booker2022extension}) extended Venkatesh’s converse theorem to holomorphic forms of arbitrary integral weight, level, and character. All of these results correspond to converse theorems for $\GL(1)$ and $\GL(2)$ automorphic forms.

  In this paper, we use the ideas of Booker, Farmer, and Lee to prove a converse theorem for half-integral weight holomorphic forms. These holomorphic forms correspond to automorphic representations of the metaplectic double cover of $\GL(2)$ rather than $\GL(2)$ itself. In general, the classical study of converse theorems has focused on $L$-series related to non-metaplectic automorphic forms. That is, in the entire class of automorphic forms, attention has focused virtually exclusively on the $n = 1$ case of automorphic forms on the $n$-fold cover of $\GL(r)$ and other groups. This is for a very simple reason: In the $n = 1$ case, automorphic representations factor, making it possible to study the global question locally at individual primes.   For $n \ge 2$ this is not the case. The corresponding representations do \emph{not} factor. Similarly, the $L$-series associated to these automorphic forms do not factor into an Euler product and do not satisfy the Riemann Hypothesis.

  This omission has been justified by several influential voices in the field by the explanation that $L$-series without Euler products are simply not interesting. This is reminiscent of Hecke deciding that the theory of Hecke operators on half-integral weight forms, that is, automorphic forms on the double cover of $\GL(2)$, was not interesting because only operators of the form $T_{p^2}$ were not identically zero. This observation was turned on its head by Shimura's paper (see \cite{shimura1973modular}) establishing a correspondence between automorphic forms on the double cover of $\GL(2)$ that were eigenfunctions of all the $T_{p^2}$ operators, and non-metaplectic ($n = 1$) forms on $\GL(2)$.   Very interestingly, for any fixed square free $t$, Shimura's work related the sequence of Fourier coefficients  $(a(tn^{2}))_{n = 1}^\infty$ of a half-integral weight form $f$ to the $L$-series of and integral weight form $F$, but left mysterious the significance, if any, of the coefficients $a(t)$ for square free $t$. In 1985, however, Waldspurger (see \cite{Waldspurger1985}) proved that up to a product of local constants $|a(t)|^2$ was equal to $L\left(\frac{1}{2},F \x \chi_{t}\right)$, the $L$-series of $F$ at the center of the critical strip, twisted by the quadratic character of conductor $t$. A generalization of this result to automorphic forms on the double cover of $\mathrm{GSp}(2r)$ has been achieved, but an extension to any other cover or group has remained a mystery.

  The point is, rather than being uninteresting by virtue of not having Euler products, $L$-series associated to automorphic forms on the double cover of $GL(2)$ have proved to have a profound relationship with non-metaplectic automorphic forms on $GL(2)$. The general consensus among the same influential voices is that these forms are interesting.

  Let's consider now the symmetric square lifting of a non-metaplectic automorphic form on $\GL(2)$ to $\GL(3)$. The $L$-series of this is a factor of the Rankin-Selberg convolution of this form with itself. The study of symmetric power liftings is at the heart of the Langlands program, and progress in this direction has, to a large extent, been made by converse theorems. Most notable here is the proof by Kim (see \cite{kim2003functoriality}) that the symmetric cube lifting from $\GL(2)$ to $\GL(4)$ is automorphic, which has led, as an immediate consequence, to the best progress so far toward the Ramanujan-Petersson Conjecture. One might ask: what is the analog of symmetric power liftings for metaplectic automorphic forms? The answer is that nothing whatsoever is known, but there are some tantalizing hints that the question is interesting.

  In the non-metaplectic case, the Rankin-Selberg convolution of a $\GL(2)$ form $f$ with itself is the product of the zeta function times the symmetric square $L$-series of $f$, which is also the $L$-series of an automorphic form on $\GL(4)$, namely an Eisenstein series formed by inducing the symmetric square lift of $f$ on $\GL(3)$, up to $\GL(4)$.

  In the case of the double cover of $\GL(2)$, the Rankin-Selberg convolution of a half-integral weight form with itself is an $L$-series which  does not factor into a symmetric square $L$-series times a zeta function, but which nevertheless has some very interesting properties. If any analog of the Langlands program were to carry through to symmetric power liftings, this $L$-series, although lacking an Euler product, should correspond to some automorphic form on some group and in fact it does. By Waldspurger's result, the Rankin-Selberg convolution of a half-integral weight form $f$ with itself is a Dirichlet series of roughly the form
  \[
    \sum_{d \ge 1}\frac{L\left(\frac{1}{2},F \x \chi_{d}\right)}{d^{w}},
  \]
  where up to an explicit finite correction polynomial, if $d = d_{0}d_{1}^{2})$, with $d_{0}$ square free, $L\left(\frac{1}{2},F \x \chi_{d}\right) = P(d)L\left(\frac{1}{2},F \x \chi_{d_{0}}\right)$. In fact, this Dirichlet series is a special value of the spin $L$-function of an Eisenstein series on the double cover of $\mathrm{GSp}(4)$ formed by inducing $F$ to this group. If this could be proved directly, via a converse theorem, it would provide a completely different proof of Waldspurger's result. The original proof involves theta liftings and representation theory.

  Converse theorems on $\GL(n)$, such as Cogdell and Piatetski-Shapiro's work in 1994 (see \cite{cogdell1994converse}), and the proofs of the automorphy of the symmetric square and cube liftings have been vital ingredients in the most significant progress made to date in the Langlands program, that is, the $n = 1$ case of the study of automorphic forms on $n$-fold covers of groups. The study of converse theorems for metaplectic forms has great potential and this paper can be viewed as an attempt to begin to fill this gap in the literature.

  The paper is structured as follows. In the introduction, we set our notation and give the statement of the main theorem. In the second section, we quote two lemmas that we need to use for the proof of the main theorem. In the third section, we prove the main theorem. In the appendicies, we derive the functional equation for the $L$-function of a half-integral weight holomorphic form twisted by an additive character and a variant of a result of Hecke used for classical converse theorems.

  Let $f(z)$ be a half-integral weight holomorphic form on $\G_{0}(4N)\backslash\H$ of weight $\l$ and character $\chi$. Under any $\g = \begin{psmallmatrix} a & b \\ c & d \end{psmallmatrix} \in \G_{0}(4N)$, $f(\g z)$ transforms as
  \[
    f(\g z) = \chi(d)\e_{d}^{-2\l}\legendre{c}{d}(cz+d)^{\l}f(z),
  \]
  where $\tlegendre{c}{d}$ is the modified Jacobi symbol as given in \cite{shimura1973modular} and $\e_{d} = 1,i$ according to whether $d \equiv 1,3 \tmod{4}$ respectively. The factor $\e_{d}^{-1}\tlegendre{c}{d}(cz+d)^{\frac{1}{2}}$ is called the theta multiplier. Setting $e(nz) = e^{2\pi inz}$, $f(z)$ admits a Fourier expansion of the form
  \[
    f(z)=\sum_{n \ge 1}f_{n}n^{\frac{\lambda-1}{2}}e(nz).
  \]
  For a rational $\a = \frac{a}{c}$, we define the complete additive twist of $f$, $\L_{f}(s,\a)$, by
  \[
    \L_{f}(s,\a) = \G_{\C}\left(s+\frac{\l-1}{2}\right)\sum_{n \ge 1}\frac{f_{n}e(n\a)}{n^{s}},
  \]
  where we define $\G_{\C}(s) = 2(2\pi)^{-s}\G(s)$. Furthermore, for any $c\in 4N\zz_{>0}$, and $a,\overline{a}\in\zz$ such that $a\overline{a}\equiv 1\pmod{c}$, the completion will satisfy the functional equation (see \cref{append:functional_equation_additive_twist})
  \[
    \Lambda_f\left(s,\frac{a}{c}\right)=i^\lambda\chi(\overline{a})\epsilon_{a}^{-2\lambda}\legendre{c}{a}c^{1-2s}\Lambda_f\left(1-s,-\frac{\overline{a}}{c}\right).
  \]
  These additively twisted functional equations are the main analytic ingredient in the following converse theorem:

  \begin{theorem}\label{thm:ConverseTheorem}
    Let $N \ge 1$, $\chi$ be a Dirichlet character modulo $4N$, $\{f_n\}_{n\geq 1}$ be a sequence of complex numbers, $\gamma(s)$ be a complex function, and $\w$ be a nonzero complex number. Given any $\alpha\in \qq$, define the complete additive twist $L$-series
    \[
    \Lambda_f(s,\alpha)=\gamma(s)\sum_{n\geq 1} \frac{f_n e(n\alpha)}{n^s}.
    \]
    Suppose the following properties are satisfied:
    \begin{enumerate}
      \item $\sum_{n\geq 1}f_n n^{-s}$ converges absolutely in the half-plane $\Re(s)>1$.
      \item $\gamma(s)=Q^s\prod_{j=1}^r \Gamma(\lambda_j s+\mu_j)$ with $Q,\lambda_j\in\rr_{>0}$ and $\mu_j\in\cc$ such that $\Re(\mu_j)>-\frac{1}{2}\lambda_j$ and $\sum_{j=1}^r\lambda_j=1$.
      \item For every $c\in 4N\zz_{>0}$ and for every pair of integers $a,\overline{a}\in\zz$ such that $a\overline{a}\equiv 1\pmod{c}$, the $L$-series $\Lambda_f\left(s,\frac{a}{c}\right)$ and $\Lambda_f\left(s,-\frac{\overline{a}}{c}\right)$ continue to entire functions of finite order and satisfy the functional equation
    \begin{equation}\label{equ:functional_equation_assumption}
      \L_{f}\left(s,\frac{a}{c}\right) = \w\chi(\conj{a})\e_{a}^{-\d}\legendre{c}{a}c^{1-2s}\L_{f}\left(1-s,-\frac{\conj{a}}{c}\right),
    \end{equation}
    for some fixed odd $\d \tmod{4}$.
    \end{enumerate}
    Then there exists a half-integer $\l$ with $2\l \equiv \d \tmod{4}$ such that if we set
    \[
      f(z) = \sum_{n\geq 1}f_nn^{\frac{\lambda-1}{2}}e(nz),
    \]
    then $f(z) \in \mc{S}_{\l}(\G_{0}(4N),\chi)$.
  \end{theorem}
\section{Lemmas}
  The proof of \cref{thm:ConverseTheorem} uses the same lemmas as found in \cite{booker2022extension}. We restate the lemmas below without proof. Furthermore, a proof of the first lemma can be found in \cite{booker2022extension} while a proof of the second lemma can be found in \cite{hoffstein2021first}.

  \begin{lemma}\label{lemma:GammaFactor}
    Let $\gamma(s)$ be as described in \cref{thm:ConverseTheorem}, and suppose that $\gamma(s)$ has poles at all but finitely many nonpositive integers. Then $\gamma(s) = CP(s)H^s\G_{\cc}(s)$, where $C,H\in\rr_{>0}$ and $P(s)$ is a monic polynomial whose roots are distinct nonpositive integers.
  \end{lemma}

  \begin{lemma}\label{lemma:SumOfRamanujanSums}
    Let $r(n;q)$ be the Ramanujan sum. Then for $n,N>0$ and $\Re(s)>1$, we have 
    \[
    \sum_{\substack{q\geq 1\\ 4N\mid q}}\frac{r(n;q)}{q^{2s}}=\begin{cases}
    \frac{\sigma_{1-2s}(n;4N)}{\zeta^{(4N)}(2s)} & n\neq 0,\\
    (4N)^{1-2s}\prod_{p\mid 4N}(1-p\inv)\frac{\z(2s-1)}{\z^{(4N)}(2s)} & n=0,
    \end{cases}
    \]
    where if $\frac{4N}{\prod_{p\mid 4N}p}\mid n$,
    \[
    \sigma_s(n;4N)=\prod_{\substack{p\mid n\\ p\nmid 4N}}\frac{p^{(\ord_p(n)+1)s}-1}{p^s-1}\cdot \prod_{p\mid 4N}\left(\frac{(1-p^{s-1})p^{(\ord_p(n)+1)s}-(1-p\inv)p^{\ord_p(4N)s}}{p^s-1}\right),
    \]
    and otherwise, $\s_s(n;4N)=0$.
  \end{lemma}
\section{The Main Argument}
  The line of argument in proving \cref{thm:ConverseTheorem} can be reduced to a sequence of steps. First, we use the Petersson trace formula to relate the Dirichlet series $\sum_{n \ge 1}f_{n}n^{-s}$ to a sum of Rankin-Selberg convolutions $L(s,f \x \conj{g})$ where $g$ ranges over an orthonormal basis of eigenforms. Rewriting the geometric side of the trace formula using the functional equations in \cref{equ:functional_equation_assumption}, we will be able to meromorphicially continue the Rankin-Selberg convolutions to the region $\Re(s) > \frac{1}{2}$ with a possible pole at $s = 1$. If the weight $\l$ (with $2\l \equiv \d \tmod{4}$) is not too small, taking a residue at $s = 1$ will show that each $f_{n}$ is linear combination of the $n$-th Fourier coefficients of the eigenbasis and our result follows. In the case the weight is small, we reach an obstruction as there is no pole. In this case, we rely on a variant of Hecke's converse theorem in half-integral weight to show that the $f_{n}$ are the Fourier coefficients of a holomorphic form.

  Now for the initial setup. Fix a half-integer $\l$ and let $\mc{H}_{\l}(\G_{0}(4N),\chi)$ be an orthonormal basis for the space $\mc{S}_{\l}(\G_{0}(4N),\chi)$. For $g \in \mc{H}_{\l}(\G_{0}(4N),\chi)$, denote its Fourier expansion by
  \[
    g(z) = \sum_{n \ge 1}\rho_{g}(n)n^{\frac{\l-1}{2}}e(nz).
  \]
  For fixed $n,m \ge 1$, the (half-integral weight) Petersson trace formula for weight $\l$ is
  \[
    \frac{\G\left(\l-1\right)}{(4\pi)^{\l-1}}\sum_{g \in \mc{H}_{\l}(\G_{0}(4N),\chi)}\rho_{g}(n)\conj{\rho_{g}(m)} = \d_{n,m}+\sum_{\substack{c \ge 1 \\ 4N \mid c}}\frac{2\pi i^{-\l}}{c}J_{\l-1}\left(\frac{4\pi\sqrt{nm}}{c}\right)S_{\chi,\l}(m,n;c),
  \]
  where $S_{\chi,\l}(m,n;c)$ is the Salie sum with theta multiplier:
  \[
    S_{\chi,\l}(m,n;c) = \sum_{\substack{a \tmod{c} \\ (a,c) = 1}}\chi(a)\legendre{a}{c}\e_{c}^{2\l}e\left(\frac{am+\conj{a}n}{c}\right).
  \]  
  Multiply the spectral side by $\z^{(4N)}(2s)\frac{f_{m}}{m^{s}}$, sum over $m$, and define $K_{n}(s,f,\chi)$ to be the result:
  \begin{equation}\label{equ:K_function_spectral_side}
    \begin{aligned}
      K_{n}(s,f,\chi) &= \z^{(4N)}(2s)\sum_{m \ge 1}\frac{f_{m}}{m^{s}}\frac{\G\left(\l-1\right)}{(4\pi)^{\l-1}}\sum_{g \in \mc{H}_{\l}(\G_{0}(4N),\chi)}\rho_{g}(n)\conj{\rho_{g}(m)} \\
      &= \frac{\G\left(\l-1\right)}{(4\pi)^{\l-1}}\sum_{g \in \mc{H}_{\l}(\G_{0}(4N),\chi)}\rho_{g}(n)L(s,f \x \conj{g}),
    \end{aligned}
  \end{equation}
  where we set
  \[
    L(s,f \x \conj{g}) = \z^{(4N)}(2s)\sum_{m \ge 1}\frac{f_{m}\conj{\rho_{g}(m)}}{m^{s}}.
  \]
  As $\sum_{m \ge 1}\frac{f_{m}}{m^{s}}$ is holomorphic for $\Re(s) > 1$ by assumption, $\sum_{m \le X}|f_{m}|^{2} \ll_{\e} X^{1+\e}$ for some $\e > 0$. Since $L(s,g \x \conj{g})$ has a pole at $s = 1$ we have by Landau's theorem an analogous bound $\sum_{m \le X}|\rho_{g}(m)|^{2} \ll_{\e} X^{1+\e}$. By Cauchy-Schwarz, these averages estimates together imply $\sum_{m \le X}|f_{m}\conj{\rho_{g}(m)}| \ll_{\e} X^{1+\e}$ which further implies $L(s,f \x \conj{g})$ is absolutely convergent for $\Re(s) > 1$. In particular, $K_{n}(s,f,\chi)$ is holomorphic for $\Re(s) > 1$.
  
  It will be convienent to setup some notation for a certain integral involved throughout the proof of \cref{thm:ConverseTheorem}. For $\l \ge \frac{9}{2}$, $x > 0$, $\Re(s) \in \left(\frac{1}{2},\frac{\l-1}{2}\right)$, and $\s_{1} \in \left(\frac{1-\l}{2},-\Re(s)\right)$, set
  \[
    F_{\l}(s,x) = \frac{1}{2\pi i}\int_{\Re(u) = \s_{1}}\frac{\G_{\C}\left(u+\frac{\l-1}{2}\right)\g\left(1-s-u\right)}{\G_{\C}\left(-u+\frac{\l+1}{2}\right)\g\left(s+u\right)}x^{u}\,du.
  \]
  Our bounds for $\Re(s)$ and $\s_{1}$ ensure that the contour avoids the poles of $\G_{\C}$-factor and $\g$-factor in the numerator. Moreover, from Stirling's formula we have
  \begin{equation}\label{equ:gamma_ratio_estimate}
    \frac{\G_{\C}\left(u+\frac{\l-1}{2}\right)\g\left(1-s-u\right)}{\G_{\C}\left(-u+\frac{\l+1}{2}\right)\g\left(s+u\right)} \ll |u|^{-2\Re(s)}.
  \end{equation}
  Indeed, the ratio of the $\G_{\C}$-factors is at most $O(|u|^{-2\Re(u)})$ and the ratio of the $\g$-factors is at most $O(|u|^{-2\Re(s+u)})$ because $\sum_{j = 1}^{r}\l_{j} = 1$. So \cref{equ:gamma_ratio_estimate} implies
  \begin{equation}\label{equ:F_int_bound}
    F_{\l}(s,x) \ll x^{\s_{1}}\frac{1}{2\pi i}\int_{\Re(u) = \s_{1}}|u|^{-2\Re(s)}\,du \ll x^{\s_{1}}.
  \end{equation}
  This estimate shows $F_{\l}(s,x)$ is absolutely bounded for $\Re(s) > \frac{1}{2}$ and fixed $x$. By manipulating the geometric side of the trace formula, we can show that $K_{n}(s,f,\chi)$ admits meromorphic continuation to the same region.

  \begin{proposition}\label{prop:analytic_continuation_of_K_function}
    For $\l \ge \frac{9}{2}$ with $2\l \equiv \d \tmod{4}$, $K_{n}(s,f,\chi)$ admits meromorphic continuation to the region $\Re(s) > \frac{1}{2}$ with at most a simple pole at $s = 1$.
  \end{proposition}
  \begin{proof}
    Restrict to $\Re(s) > \frac{5}{4}$. Looking at the geometric side of the Petersson trace formula,
    \[
      K_{n}(s,f,\chi) = \z^{(4N)}(2s)\frac{f_{n}}{n^{s}}+\z^{(4N)}(2s)\sum_{m \ge 1}\frac{f_{m}}{m^{s}}\sum_{\substack{c \ge 1 \\ 4N \mid c}}\frac{2\pi i^{-\l}}{c}J_{\l-1}\left(\frac{4\pi\sqrt{nm}}{c}\right)S_{\chi,\l}(m,n;c).
    \]
    Recall the well-known estimates
    \[
      J_{\l-1}(y) \ll \min\left\{y^{\l-1},y^{-\frac{1}{2}}\right\} \quad \text{and} \quad S_{\chi, \l}(m,n;c) \ll_{m,n,\e}c^{\frac{1}{2}+\e}.
    \]
    Breaking the sum over $m$ and $c$ according to which value dominates, we have
    \[
      \sum_{m \ll c}\frac{f_{m}}{m^{s}c}J_{\l-1}\left(\frac{4\pi\sqrt{nm}}{c}\right)S_{\chi,\l}(m,n;c)+\sum_{m \gg c}\frac{f_{m}}{m^{s}c}J_{\l-1}\left(\frac{4\pi\sqrt{nm}}{c}\right)S_{\chi,\l}(m,n;c).
    \]
    Applying the $J$-Bessel bound $\ll y^{\l-1}$ to the first sum and $\ll y^{-\frac{1}{2}}$ to the second (along with the Weil bound for both sums) the expression above is no larger than $O\left(\sum_{m \ge 1}f_{m}m^{\frac{1}{4}-s}\right)$ which is absolutely convergent for $\Re(s) > \frac{5}{4}$. Therefore we may interchange the sums resulting in
    \[
      K_{n}(s,f,\chi) = \z^{(4N)}(2s)\frac{f_{n}}{n^{s}}+\z^{(4N)}(2s)2\pi i^{-\l}\sum_{\substack{c \ge 1 \\ 4N \mid c}}\frac{1}{c}\sum_{m \ge 1}\frac{f_{m}S_{\chi,\l}(m,n;c)}{m^{s}}J_{\l-1}\left(\frac{4\pi\sqrt{nm}}{c}\right).
    \]
    Now restrict to the region $\Re(s) \in \left(\frac{5}{4},\frac{\l-1}{2}\right)$. Recall the Mellin-Barnes integral repersentation for the $J$-Bessel function which is valid for $\s_{0} \in \left(\frac{1-\l}{2},0\right)$:
    \[
      J_{\l-1}(4\pi y) = \frac{1}{4\pi^{2}i}\int_{\Re(u) = \s_{0}}\frac{\G_{\C}\left(u+\frac{\l-1}{2}\right)}{\G_{\C}\left(-u+\frac{\l+1}{2}\right)}y^{-2u}\,du.
    \]
    In particular, for $\s_{0} \in (1-\Re(s),0)$ we can interchange the integral and sum over $m$ to obtain
    \begin{equation}\label{equ:intermediate_sum_for_c}
      \z^{(4N)}(2s)\frac{f_{n}}{n^{s}}+\z^{(4N)}(2s)i^{-\l}\sum_{\substack{c \ge 1 \\ 4N \mid c}}\frac{1}{2\pi i}\int_{\Re(u) = \s_{0}}\frac{\G_{\C}\left(u+\frac{\l-1}{2}\right)}{\G_{\C}\left(-u+\frac{\l+1}{2}\right)}c^{2u-1}\sum_{m \ge 1}\frac{f_{m}S_{\chi,\l}(m,n;c)}{m^{s}}(nm)^{-u}\,du.
    \end{equation}
    Opening up the Salie sum and interchanging it with the sum over
    $m$, the outside sum over $c$ is
    \[
      \sum_{\substack{c \ge 1 \\ 4N \mid c}}\frac{1}{2\pi i}\int_{\Re(u) = \s_{0}}\frac{\G_{\C}\left(u+\frac{\l-1}{2}\right)}{\G_{\C}\left(-u+\frac{\l+1}{2}\right)}c^{2u-1}n^{-u}\sum_{\substack{a \tmod{c} \\ (a,c) = 1}}\chi(a)\e_{a}^{2\l}\legendre{c}{a}e\left(n\frac{\conj{a}}{c}\right)L_{f}\left(s+u,\frac{a}{c}\right)\,du,
    \]
    where
    \[
      L_{f}\left(s,\frac{a}{c}\right) = \frac{\L_{f}\left(s,\frac{a}{c}\right)}{\g(s)} = \sum_{m \ge 1}\frac{f_{m}e\left(m\frac{a}{c}\right)}{m^{s}}.
    \]
    The next step is to apply the functional equation. First, shift the contour to $\s_{1} \in (-\frac{\l-1}{2},-\Re(s))$. We do not pick up residues since the first pole occurs at $u = -\frac{\l-1}{2}$ from $\G_{\C}\left(u+\frac{\l-1}{2}\right)$. Upon substiuting \cref{equ:functional_equation_assumption}, our assumption that $2\l \equiv \d \tmod{4}$ yields
    \begin{equation}\label{equ:c_sum_after_fun_eq}
      \begin{aligned}
        \w\sum_{\substack{c \ge 1 \\ 4N \mid c}}\frac{1}{c^{2s}}\frac{1}{2\pi i}&\int_{\Re(u) = \s_{1}}\frac{\G_{\C}\left(u+\frac{\l-1}{2}\right)\g(1-s-u)}{\G_{\C}\left(-u+\frac{\l+1}{2}\right)\g(s+u)}n^{-u} \\
        &\cdot \sum_{\substack{a \tmod{c} \\ (a,c) = 1}}e\left(n\frac{\conj{a}}{c}\right)L_{f}\left(1-s-u,-\frac{\conj{a}}{c}\right)\,du.
      \end{aligned}
    \end{equation}
    Now $\Re(1-s-u) > 1$ so that $L\left(1-s-u,-\frac{\conj{a}}{c}\right) \ll 1$. Expanding the $L$-series as a Dirichlet series and interchanging with the sum over $a$, we have
    \begin{equation}\label{equ:L-series_Ramanujan}
      \sum_{\substack{a \tmod{c} \\ (a,c) = 1}}e\left(n\frac{\conj{a}}{c}\right)L_{f}\left(1-s-u,-\frac{\conj{a}}{c}\right) = \sum_{m \ge 1}\frac{f_{m}r(n-m;c)}{m^{1-s-u}}.
    \end{equation}
    where $r(n;c)$ is the Ramanujan sum. Inserting \cref{equ:L-series_Ramanujan} into \cref{equ:c_sum_after_fun_eq} and replacing the result with the sum over $c$ in \cref{equ:intermediate_sum_for_c} yields
    \begin{align*}
      K_{n}(s,f,\chi) &= \z^{(4N)}(2s)\frac{f_{n}}{n^{s}}+\z^{(4N)}(2s) \\
      &\cdot \w i^{-\l}\sum_{\substack{c \ge 1 \\ 4N \mid c}}\frac{1}{c^{2s}}\frac{1}{2\pi i}\int_{\Re(u) = \s_{1}}\frac{\G_{\C}\left(u+\frac{\l-1}{2}\right)\g\left(1-s-u\right)}{\G_{\C}\left(-u+\frac{\l+1}{2}\right)\g\left(s+u\right)}n^{-u}\sum_{m \ge 1}\frac{f_{m}r(n-m;c)}{m^{1-s-u}}\,du.
    \end{align*}
    Since $L\left(1-s-u,-\frac{\conj{a}}{c}\right) \ll 1$, \cref{equ:gamma_ratio_estimate,equ:L-series_Ramanujan} imply that the integrand is absolutely bounded so we may interchange the integral and sum over $m$ to obtain
    \[
      K_{n}(s,f,\chi) = \z^{(4N)}(2s)\frac{f_{n}}{n^{s}}+\z^{(4N)}(2s)\w i^{-\l}\sum_{\substack{c \ge 1 \\ 4N \mid c}}\frac{1}{c^{2s}}\sum_{m \ge 1}\frac{f_{m}r(n-m;c)}{m^{1-s}} F_{\l}\left(s,\frac{m}{n}\right).
    \]
    In particular, $F_{\l}\left(s,\frac{m}{n}\right) \ll_{n} m^{\s_{1}}$ by \cref{equ:F_int_bound}, so \cref{equ:L-series_Ramanujan} implies the sum over $m$ is absolutely bounded. Therefore the sum over $c$ is absolutely bounded too. Interchanging the sums over $c$ and $m$ yields
    \[
      K_{n}(s,f,\chi) = \z^{(4N)}(2s)\frac{f_{n}}{n^{s}}+\z^{(4N)}(2s)\w i^{-\l}\sum_{m \ge 1}\frac{f_{m}}{m^{1-s}}\sum_{\substack{c \ge 1 \\ 4N \mid c}}\frac{r(n-m;c)}{c^{2s}} F_{\l}\left(s,\frac{m}{n}\right).
    \]
    Using \cref{lemma:SumOfRamanujanSums}, we can compute the sum over $c$ to obtain
    \begin{equation}\label{equ:K_geometric_side_analytic_continuation}
      \begin{aligned}
        K_{n}(s,f,\chi) = \z^{(4N)}(2s)\frac{f_{n}}{n^{s}}&+\w i^{-\l}\frac{f_{n}}{n^{1-s}}\z(2s-1)(4N)^{1-2s}\prod_{p \mid 4N}(1-p^{-1})F_{\l}(s,1) \\
        &+\w i^{-\l}\sum_{\substack{m \ge 1 \\ m \neq n}}\frac{f_{m}\s_{1-2s}(n-m;4N)}{m^{1-s}}F_{\l}\left(s,\frac{m}{n}\right).
      \end{aligned}
    \end{equation}
    We now argue that \cref{equ:K_geometric_side_analytic_continuation} gives meromorphic continuation to the region $\frac{1}{2} < \Re(s) < -\s_{1}$. Indeed, we have already remarked that $F_{\l}(s,1)$ is absolutely bounded. As for the sum over $m$, we use the estimates $F_{\l}\left(s,\frac{m}{n}\right) \ll_{n} m^{\s_{1}}$ by \cref{equ:F_int_bound} and $\s_{1-2s}(n-m;N) \ll_{N,n,\e} m^{\e}$ for all $m \neq n$ so that
    \[
      \sum_{\substack{m \ge 1 \\ m \neq n}}\frac{f_{m}\s_{1-2s}(n-m;N)}{m^{1-s}}F_{\l}\left(s,\frac{m}{n}\right) \ll_{n} \sum_{\substack{m \ge 1 \\ m \neq n}}\frac{f_{m}\s_{1-2s}(n-m;N)m^{\s_{1}}}{m^{1-s}} \ll_{N,n,\e} \sum_{\substack{m \ge 1 \\ m \neq n}}\frac{f_{m}}{m^{1-s-\e-\s_{1}}},
    \]
    and this latter sum converges absolutely for $\Re(s) < -\e-\s_{1}$. Therefore all the terms in the right-hand side of \cref{equ:K_geometric_side_analytic_continuation} are meromorphic for $\frac{1}{2} < \Re(s) < -\s_{1}$ with a possible simple pole at $s = 1$ coming from $\z(2s-1)$ in the middle term. The meromorphic continuation of $K_{n}(s,f,\chi)$ is now established.
  \end{proof}

  The second result is that if $F_{\l}(1,1)$ vanishes for all sufficiently large $\l$ with $2\l \equiv \d \tmod{4}$, then $\g(s)$ can be specified explicitely.
  \begin{proposition}\label{prop:specify_gamma_factor_if_integral_vanishes}
    Let $\g(s)$ be given as in \cref{thm:ConverseTheorem} and suppose there is a residue class $\d \tmod{4}$ such that $F_{\l}(1,1) = 0$ for all $\l \ge \frac{9}{2}$ with $2\l \equiv \d \tmod{4}$. Then there exists $\nu \in \left\{\frac{1}{2},\frac{3}{2},\frac{5}{2},\frac{7}{2}\right\}$ such that $2\nu \equiv \d \tmod{4}$ with $\g(s) = CH^{s}\G_{\C}\left(s+\frac{\nu-1}{2}\right)$.
  \end{proposition}
  \begin{proof}
    Make the change of variables $u \to \frac{u}{2}$ in $F_{\l}(1,1)$ and shift the contour to $\Re(u) = -\frac{5}{2}$ to obtain
    \[
      F_{\l}(1,1) = \frac{1}{2\pi i}\int_{\Re(u) = -\frac{5}{2}}\frac{\G_{\C}\left(\frac{\l-1+u}{2}\right)\g\left(-\frac{u}{2}\right)}{2\G_{\C}\left(\frac{\l+1-u}{2}\right)\g\left(1+\frac{u}{2}\right)}\,du = 0,
    \]
    for all $\l \ge \frac{9}{2}$ with $2\l \equiv \d \tmod{4}$. Now consider the function
    \[
      f_{\l}(y) = \frac{\chi_{(0,1)}(y)}{\sqrt{1-y^{2}}}\cos(\l\arccos(y)).
    \]
    Upon changing variables $y \to \cos(y)$, the Mellin transform of $f_{\l}(y)$ is (see \cite{zwillinger2007table} \S3.631 \#9):
    \[
      \wtilde{f_{\l}}(s) = \int_{0}^{\infty}f_{\l}(y)y^{s-1}\,dy = \frac{\G_{\C}(s)}{2^{s}\G_{\C}\left(\frac{s+\l+1}{2}\right)\G_{\C}\left(\frac{s-\l+1}{2}\right)},
    \]
    for $\Re(s) > 0$. Let $\conj{\nu} \in \left\{\frac{1}{2},\frac{3}{2},\frac{5}{2},\frac{7}{2}\right\}$ be such that $2\conj{\nu}+2\l \equiv 4 \tmod{8}$. Then, making use of Euler's reflection formula, we have
    \begin{align*}
      \frac{\G_{\C}\left(\frac{\l-1+u}{2}\right)\g\left(-\frac{u}{2}\right)}{2\G_{\C}\left(\frac{\l+1-u}{2}\right)\g\left(1+\frac{u}{2}\right)} &= \frac{\G_{\C}(1-u)}{2^{1-u}\G_{\C}\left(\frac{\l+1-u}{2}\right)\G_{\C}\left(\frac{3-\l-u}{2}\right)}\frac{2^{-u}\G_{\C}\left(\frac{3-\l-u}{2}\right)\G_{\C}\left(\frac{\l-1+u}{2}\right)}{\G_{\C}(1-u)}\frac{\g\left(-\frac{u}{2}\right)}{\g\left(1+\frac{u}{2}\right)} \\
      &= \wtilde{f_{\l-1}}(1-u)\frac{2^{-u}\G_{\C}\left(\frac{3-\l-u}{2}\right)\G_{\C}\left(\frac{\l-1+u}{2}\right)}{\G_{\C}(1-u)}\frac{\g\left(-\frac{u}{2}\right)}{\g\left(1+\frac{u}{2}\right)} \\
      &= \wtilde{f_{\l-1}}(1-u)\frac{2^{1-u}}{\sin\left(\frac{\pi}{2}(\l-1+u)\right)\G_{\C}(1-u)}\frac{\g\left(-\frac{u}{2}\right)}{\g\left(1+\frac{u}{2}\right)} \\
      &= \wtilde{f_{\l-1}}(1-u)\frac{2^{1-u}}{\sin\left(\frac{\pi}{2}(u+1-\conj{\nu})\right)\G_{\C}(1-u)}\frac{\g\left(-\frac{u}{2}\right)}{\g\left(1+\frac{u}{2}\right)} \\
      &= \wtilde{f_{\l-1}}(1-u)\frac{2^{-u}\G_{\C}\left(\frac{u+1-\conj{\nu}}{2}\right)\G_{\C}\left(\frac{1-u+\conj{\nu}}{2}\right)}{\G_{\C}(1-u)}\frac{\g\left(-\frac{u}{2}\right)}{\g\left(1+\frac{u}{2}\right)},
    \end{align*}
    where the fourth equality follows because $2\conj{\nu}+2\l \equiv 4 \tmod{8}$. Set
    \[
      \wtilde{g}(u) = \frac{2^{-u}\G_{\C}\left(\frac{u+1-\conj{\nu}}{2}\right)\G_{\C}\left(\frac{1-u+\conj{\nu}}{2}\right)}{\G_{\C}(1-u)}\frac{\g\left(-\frac{u}{2}\right)}{\g\left(1+\frac{u}{2}\right)},
    \]
    so that we can express $F_{\l}(1,1)$ as
    \[
      F_{\l}(1,1) = \frac{1}{2\pi i}\int_{\Re(u) = -\frac{5}{2}}\wtilde{f_{\l-1}}(1-u)\wtilde{g}(u)\,du = 0.
    \]
    Now $\wtilde{g}(u)$ is holomorphic in a strip containing $\Re(u) = -\frac{5}{2}$ since at $u = -\frac{5}{2}$ the gamma factors $\G_{\C}\left(\frac{u+1-\conj{\nu}}{2}\right)$, $\G_{\C}\left(\frac{1-u+\conj{\nu}}{2}\right)$, and $\g\left(-\frac{u}{2}\right)$ are all holomorphic (for all $\conj{\nu}$). Moreover, by Stirling's formula we have that $\wtilde{g}(u) \ll |u|^{-\frac{3}{2}}$ and so
    \begin{equation}\label{equ:Mellin_transform_bound_for_g}
      \int_{\Re(u) = -\frac{5}{2}}\wtilde{g}(u)y^{-u}\,du \ll y^{\frac{5}{2}}\int_{\Re(u) = -\frac{5}{2}}|u|^{-\frac{3}{2}}\,du \ll y^{\frac{5}{2}}.
    \end{equation}
    Therefore the Mellin inverse of $\wtilde{g}(u)$ exsits. In particular,
    \begin{equation}\label{equ:g_as_Mellin_transform}
      g(y) = \frac{1}{2\pi i}\int_{\Re(u) = -\frac{5}{2}}\wtilde{g}(u)y^{-u}\,du,
    \end{equation}
    is a continuous function on $[0,\infty)$. Our goal now is to show that $g$ is identically zero on $[0,1]$. By Fubini's theorem,
    \begin{equation}\label{equ:int_of_f_and_g_is_zero}
      0 = \frac{1}{2\pi i}\int_{\Re(u) = -\frac{5}{2}}\wtilde{f_{\l-1}}(1-u)\wtilde{g}(u)\,du = \int_{0}^{\infty}f_{\l-1}(y)g(y)\,dy.
    \end{equation}
    Making the change of variables $y = \cos(\t)$ and noting that $\cos(\t) = |\cos(\t)|$ for $0 \le \t \le \frac{\pi}{2}$, \cref{equ:int_of_f_and_g_is_zero} becomes
    \[
      \int_{0}^{\frac{\pi}{2}}\frac{\cos((\l-1)\t)}{\cos\left(\frac{\t}{2}\right)}\cos\left(\frac{\t}{2}\right)g\left(\left|\cos(\t)\right|\right)\,d\t = 0.
    \]
    Changing variables $v = \cos(\t)$, this latter integral is expressable as
    \begin{equation}\label{equ:orthogonality_with_Chebyschev}
      \int_{-1}^{1}V_{\lf\l-1\rf}(v)h(v)\sqrt{\frac{1+v}{1-v}}\,dv = 0,
    \end{equation}
    where $V_{n}$ is the $n$-th Chebyshev polynomial of the third kind and we set
    \[
      h(v) = \chi_{[0,1]}(v)\frac{g(v)}{\sqrt{1+v}}.
    \]
    Suppose $\d = 1$. Then by \cref{equ:orthogonality_with_Chebyschev}, $h(v)$ is orthogonal to all $V_{n}$ for odd $n \ge 3$. So there is a an even function $P(v)$ and constant $a$ such that $h(v) = av+P(v)$. If $v \in [0,1]$, then $h(v) = 2av$ and therefore
    \[
      h(v) = \begin{cases} 0 & v \in [-1,0), \\ 2av & v \in [0,1]. \end{cases}
    \]
    So $g(y) = 2ay\sqrt{1+y}$ for $y \in [0,1]$. By \cref{equ:Mellin_transform_bound_for_g,equ:g_as_Mellin_transform}, $g(y) = O(y^{\frac{5}{2}})$ as $y \to 0$ so that $a = 0$. Now suppose $\d = 3$. Then $h(v)$ is orthogonal to all $V_{n}$ for even $n \ge 4$. So there is an odd function $Q(v)$ and constants $b$ and $c$ such that $h(v) = b+cv^{2}+Q(v)$. If $v \in [0,1]$, we have $h(v) = 2(b+cv^{2})$ so that
    \[
      h(v) = \begin{cases} 0 & v \in [-1,0), \\ 2(b+cv^{2}) & v \in [0,1]. \end{cases}
    \]
    Then $g(y) = 2(b+cy^{2})\sqrt{1+y}$ for $y \in [0,1]$.  Since $\wtilde{g}(u)$ is holomorphic on a strip containing $\Re(u) = -\frac{5}{2}$, shifting the contour in \cref{equ:Mellin_transform_bound_for_g} to $\Re(u) = -\frac{5}{2}-\e$ for some small $\e > 0$, \cref{equ:g_as_Mellin_transform} implies $g(y) = O(y^{\frac{5}{2}+\e})$ as $y \to 0$ so that $b = a = 0$. In either case, it follows that $g$ is identically zero on $[0,1]$. Then the Mellin transform of $g$ takes the form
    \[
      \wtilde{g}(u) = \int_{1}^{\infty}g(y)y^{u-1}\,dy.
    \]
    Since $g(y) \ll y^{\frac{5}{2}}$, the analytic continuation of $\wtilde{g}(u)$ to $\Re(u) < -\frac{5}{2}$ follows. In this region, $\frac{\G_{\C}\left(\frac{1-u+v}{2}\right)\g\left(-\frac{u}{2}\right)}{\G_{\C}(1-u)}$ is analytic and nonvanishing. Since $\wtilde{g}(u)$ is analytic, we conclude $\frac{\G_{\C}\left(\frac{u+1-\conj{\nu}}{2}\right)}{\g\left(1+\frac{u}{2}\right)}$ is analytic for $\Re(u) < -\frac{5}{2}$ too. In this region, the poles of $\G_{\C}\left(\frac{u+1-\conj{\nu}}{2}\right)$ occur at $u = \conj{\nu}-1+2n$ for all $n < -\frac{3+2\conj{\nu}}{4}$. Therefore $\wtilde{\g}(s) = \g(s+\frac{\conj{\nu}+1}{2})$ has poles at all but finitely many nonpositive integers (determined in \cref{tab:table_for_proof_of_lemma}). Applying \cref{lemma:GammaFactor} we get $\wtilde{\g}(s) = C'P(s)H^{s}\G_{\C}(s)$ for some constants $C',H \in \R_{>0}$ and a monic polynomial $P(s)$. In terms of $\g(s)$, we have
     \[
      \g(s) = CP\left(s-\frac{\conj{\nu}+1}{2}\right)H^{s}\G_{\C}\left(s-\frac{\conj{\nu}+1}{2}\right),
     \]
     where $C = C'H^{-\frac{\conj{\nu}+1}{2}}$. In order to finish the theorem we need $P(s)$ to be such that 
      \[
        P\left(s-\frac{\conj{\nu}+1}{2}\right)\G_{\C}\left(s-\frac{\conj{\nu}+1}{2}\right) = \G_{\C}\left(s+\frac{\nu-1}{2}\right),
      \]
      for some $\nu \in \left\{\frac{1}{2},\frac{3}{2},\frac{5}{2},\frac{7}{2}\right\}$ with $2\nu \equiv \d \tmod{4}$. The condition $\Re(\mu_{j}) > -\frac{1}{2}\l_{j}$ implies $\g(s)$ cannot have poles in the region $\Re(s) \ge \frac{1}{2}$. In other words, $\wtilde{\g}(s)$ cannot have poles in the region $\Re(s) \ge -\frac{\conj{\nu}}{2}$. Therefore $P(s)$ needs to have a zero canceling the possible poles coming from $\G_{\C}(s)$ in $\wtilde{\g}(s) = C'P(s)H^{s}\G_{\C}(s)$. \cref{tab:table_for_proof_of_lemma} computes the possible polynomials $P(s)$ and corresponding weights $\nu$ for a given $\conj{\nu}$:

      \begin{table}[ht]
        \begin{center}
        \caption{}\label{tab:table_for_proof_of_lemma}
        \begin{stabular}[1.5]{|c|c|c|c|c|c|}
          \hline
            $\conj{\nu}$ & $\d$ & Poles of $\wtilde{\g}(s)$ & Possible poles of $\wtilde{\g}(s)$ & $P(s)$ & $\nu$ \\
            \hline
            $\conj{\nu} = \frac{1}{2}$ & 3 & $s = n \le -2$ & $s = -1$ & $s$ or $s(s+1)$ & $\nu = \frac{3}{2}$ or $\nu = \frac{7}{2}$ \\
            \hline
            $\conj{\nu} = \frac{3}{2}$ & 1 & $s = n \le -2$ & $s = -1$ & $s$ or $s(s+1)$ & $\nu = \frac{1}{2}$ or $\nu = \frac{5}{2}$ \\
            \hline
            $\conj{\nu} = \frac{5}{2}$ & 3 & $s = n \le -3$ & $s = -2$ & $s(s+1)$ or $s(s+1)(s+2)$ & $\nu = \frac{3}{2}$ or $\nu = \frac{7}{2}$ \\
            \hline
            $\conj{\nu} = \frac{7}{2}$ & 1 & $s = n \le -3$ & $s = -2$ & $s(s+1)$ or $s(s+1)(s+2)$ & $\nu = \frac{1}{2}$ or $\nu = \frac{5}{2}$ \\
            \hline
        \end{stabular}
        \end{center}
      \end{table}
      Upon inspection it is clear that for any $\d$, the possible weights $\nu$ satisfy $2\nu \equiv \d \tmod{4}$.
  \end{proof}

  We can now prove \cref{thm:ConverseTheorem}.

  \begin{proof}[Proof of \cref{thm:ConverseTheorem}]
    We divide the proof into two parts based on whether $F_{\l}(1,1)$ vanishes or not since the methods invovled are seperate:
    \begin{enumerate}[label=(\roman*)]
      \item Suppose $F_{\l}(1,1) \neq 0$ for some $\l \ge \frac{9}{2}$ with $2\l \equiv \d \tmod{4}$. Let $d = \dim\mc{S}_{\l}(\G_{0}(4N),\chi)$. Since the $g$ form a basis for $\mc{S}_{\l}(\G_{0}(4N),\chi)$, we can choose positive integers $n_{1} < n_{2} < \cdots < n_{d}$ such that the $d$ many $d$-dimensional vectors $(\rho_{g}(n_{1}),\ldots,\rho_{g}(n_{d}))$ must be linearly independent. For otherwise, some linear combination of the basis vectors $g$ has $d$ many vanishing Fourier coefficients and hence must be identically zero contradicting the linear independence of the $g$. The definition of $K_{n_{i}}(s,f,\chi)$ is a linear combination of the Rankin-Selberg convolutions $L(s,f \x \conj{g})$. Since the vectors $(\rho_{g}(n_{1}),\ldots,\rho_{g}(n_{d}))$ are linearly independent, we can invert this linear relation so that each $L(s,f \x \conj{g})$ is a linear combination of the $K_{n_{i}}(s,f,\chi)$ for $1 \le i \le d$. By \cref{prop:analytic_continuation_of_K_function}, each $L(s,f \x \conj{g})$ admits meromorphic continuation to the region $\Re(s) > \frac{1}{2}$ with at most a simple pole at $s = 1$. Taking the residue at $s = 1$ gives
      \[
        \frac{\G\left(\l-1\right)}{(4\pi)^{\l-1}}\sum_{g \in \mc{H}_{\l}(\G_{0}(4N),\chi)}\rho_{g}(n)\Res_{s = 1}L(s,f \x \conj{g}) = \w i^{-\l}\frac{1}{2}(4N)^{-1}\prod_{p \mid 4N}(1-p^{-1})F_{\l}(1,1)f_{n},
      \]
      for every $n \ge 1$. As $F_{\l}(1,1) \neq 0$, we may isolate $f_{n}$ and conclude that $f_{n}$ is a linear combination of $n$-th Fourier coefficients of basis elements $g$. Moreover, the coefficients of this linear combination are independent of $n$. Since this linear relation holds for all $n$, $f \in \mc{S}_{\l}(\G_{0}(4N),\chi)$.
      
      \item Suppose $F_{\l}(1,1) = 0$ for all $\l\geq \frac{9}{2}$ with $2\l\equiv \d\tmod{4}$. Using \cref{prop:specify_gamma_factor_if_integral_vanishes}, we may assume that $\gamma(s)=CH^s\G_{\cc}\left(s+\frac{\nu-1}{2}\right)$ where $C,H\in\rr_{>0}$ and $\nu\in\left\{\frac{1}{2},\frac{3}{2},\frac{5}{2},\frac{7}{2}\right\}$ with $2\nu\equiv\d\tmod{4}$. We will show that $H = 1$. Suppose first that $H>1$. Then
      \[
        F_{\l}(1,1) = \frac{1}{2\pi i}\int_{\Re(u) = \s_{1}} H^{-u-1}\frac{\G_{\cc}\left(\frac{\l-1+u}{2}\right)\G_{\cc}\left(\frac{\nu-1-u}{2}\right)}{2\G_{\cc}\left(\frac{\l+1-u}{2}\right)\G_{\cc}\left(\frac{\nu+1+u}{2}\right)}\,du=0.
      \]
      Replacing the $\G_{\cc}$-factors with $\G$-factors and clearing nonzero constants, the integral becomes
      \begin{equation}\label{equ:integral_polynomial_is_zero}
        \frac{1}{2\pi i}\int_{\Re(u) = \s_{1}} H^{-u}\frac{\G\left(\frac{\l-1+u}{2}\right)\G\left(\frac{\nu-1-u}{2}\right)}{\G\left(\frac{\l+1-u}{2}\right)\G\left(\frac{\nu+1+u}{2}\right)}\,du=0.
      \end{equation}
      By Stirling's formula, the integrand in \cref{equ:integral_polynomial_is_zero} is $O(H^{-u}|u|^{-2\Re(u+1)})$. Since $H > 1$, this estimate shows that upon taking $\s_{1} \to \infty$ the integrand vanishes. Therfore we may shift the line of integration in \cref{equ:integral_polynomial_is_zero} all the way to the right by taking $\s_{1} \to \infty$ at the cost of pickup up residues. We obtain the following relation:
      \begin{equation}\label{equ:integral__residue_polynomial_relation}
        \sum_{\rho}\Res_{u = \rho}\left(H^{-u}\frac{\G\left(\frac{\l-1+u}{2}\right)\G\left(\frac{\nu-1-u}{2}\right)}{2\G\left(\frac{\l+1-u}{2}\right)\G\left(\frac{\nu+1+u}{2}\right)}\right) = 0,
      \end{equation}
      where we are summing over all poles $\rho$ of the integrand in \cref{equ:integral_polynomial_is_zero} subject to $\Re(\rho) > \s_{1} > \frac{1-\l}{2}$. For every $\d$, by \cref{prop:specify_gamma_factor_if_integral_vanishes} there are two possible weights $\nu$. For each such $(\d,\nu)$ pair we choose one or two weights $\l \ge \frac{9}{2}$. For each triple $(\d,\nu,\l)$, the residue in \cref{equ:integral__residue_polynomial_relation} is an algebraic relation in $H$:
      \[
        G_{\d,\nu,\l}(H) = 0.
      \]
      Given $(\d,\nu)$, $H$ must be a zero of of this relation for all $\l \ge \frac{9}{2}$ with $2\l \equiv \d \tmod{4}$. \cref{tab:table_for_proof_of_theorem} computes various $G_{\d,\nu,\l}(H)$ and from it we see that for any given tuple $(\d,\nu)$, the corresponding triples $(\d,\nu,\l)$ only have the root $H = 1$ in common. But this contradicts $H$ being common zero larger than $1$. 
      
              \begin{table}[ht]
        \begin{center}
        \caption{}\label{tab:table_for_proof_of_theorem}
        \centering\renewcommand\cellalign{c}
        \setcellgapes{3pt}\makegapedcells
          \begin{stabular}[1]{|c|c|c|c|c|}
          \hline $\d$ & $\nu$ & $\l$ & $G_{\d,\nu,\l}(H)$ & Values of $H$ \\

          \hline $1$ & $\frac{1}{2}$ & $\frac{9}{2}$ &  $-\frac{1}{2}H^{1/2}+3H^{-3/2}-\frac{5}{2}H^{-7/2}$ & $H=1,\sqrt{5}$ \\

          \hline $1$ & $\frac{1}{2}$ & $\frac{13}{2}$ & $-\frac{1}{4}H^{1/2}+\frac{15}{4}H^{-3/2}-\frac{35}{4}H^{-7/2}+\frac{21}{4}H^{-11/2}$& $H=1,\sqrt{7\pm2\sqrt{7}}$\\

          \hline $1$ & $\frac{5}{2}$ & $\frac{9}{2}$ & $-2H^{-3/2}+2H^{-7/2}$ & $H=1$  \\

          \hline $3$ & $\frac{3}{2}$ & $\frac{11}{2}$ & $-\frac{3}{2}H^{-1/2}+5H^{-5/2}-\frac{7}{2}H^{-9/2}$ & $H=1,\sqrt{\frac{7}{3}}$ \\

          \hline $3$ & $\frac{3}{2}$ & $\frac{15}{2}$ &  $-\frac{5}{4}H^{-1/2}+\frac{35}{4}H^{-5/2}-\frac{63}{4}H^{-9/2}+\frac{33}{4}H^{-13/2}$ & $H=1,\sqrt{\frac{1}{5}(15\pm2\sqrt{15})}$\\

          \hline $3$ & $\frac{7}{2}$ & $\frac{11}{2}$ & $-2H^{-5/2}+2H^{-9/2}$ & $H=1$\\
          \hline
        \end{stabular}
      \end{center}
      \end{table}
      
      Now assume $H \le 1$. In this case, the integrand in \cref{equ:integral_polynomial_is_zero} does not vanish as $\s_{1} \to \infty$ and the previous arguemnt cannot be used. Instead, we appeal to a variant of Hecke's converse theorem. By \cref{prop:Hecke_modularity_variant}, $f$ satisfies
      \[
        f\bigg\vert\begin{pmatrix}aH^2 & -1\\ cH^2 & 0\end{pmatrix}=\w i^{-\nu}\chi(\overline{a})\epsilon_a^{-2\nu}\legendre{c}{a}f\bigg\vert\begin{pmatrix}
        1 & -\overline{a}\\ 0 & c \end{pmatrix},
      \]
      for $c\in 4N\zz_{>0}$ and $a,\overline{a}\in\zz$ with $a\overline{a}\equiv1\pmod{c}$. We can express this action more convienently as 
    \begin{equation}\label{eq:SlashOperatorModularity}
    f\bigg\vert\begin{pmatrix}aH & \frac{a\overline{a}H^2-1}{H}\\ cH & \overline{a}H\end{pmatrix}=\w i^{-\nu}\chi(\overline{a})\epsilon_a^{-2\nu}\legendre{c}{a}f.
    \end{equation}
    The argument presented in \cite{booker2022extension} to show that $H = 1$ and $\w = i^{\nu}$ works vertabium. Thus, the matrices in \cref{eq:SlashOperatorModularity} are of the form
    \[
      M =\begin{pmatrix}
        a & \frac{a\overline{a}-1}{c}\\
        c & \overline{a}
    \end{pmatrix},
    \]
    and letting $\chi(M)=\chi(\overline{a})$, \cref{eq:SlashOperatorModularity} can be compactly expressed as
    \[
      f\vert M=\chi(M)\epsilon_a^{-2\nu}\legendre{c}{a}f.
    \]
    This is equivalent to the modularity condition for half-integral weight forms. The matrices $M$ along with $\begin{psmallmatrix}
        1 & 0\\0 & 1
    \end{psmallmatrix}$ and $\begin{psmallmatrix}
        1 & 1\\ 0 & 1
    \end{psmallmatrix}$ generate $\G_0(4N)$. Therefore, $f\in S_\lambda(\Gamma_0(4N),\chi)$ as desired.
    \end{enumerate}    
  \end{proof}

\appendix
\section{Functional Equations for Additive Twists}\label{append:functional_equation_additive_twist}
  Let $f(z)$ be a half-integral weight holomorphic form on $\G_{0}(4N)\backslash\H$ of weight $\l$ twisted by a Dirichlet character $\chi$ modulo $4N$. Let $f(z) = \sum_{n \ge 1}f_{n}n^{\frac{\l-1}{2}}e(nz)$ be the Fourier expansion of $f$. The following proposition proves the analytic continuation and function equation of the additive twist $\L_{f}(s,\frac{a}{c})$.
  
  \begin{proposition}
    For any $c \in 4N\Z_{> 0}$ and $a,\conj{a} \in \Z$ with $a\conj{a} \equiv 1 \tmod{c}$, $\L_{f}(s,\frac{a}{c})$ and $\L_{f}(s,-\frac{\conj{a}}{c})$ admit analytic continuation to $\C$ and satisfy the functional equation
    \begin{equation}
      \L_{f}\left(s,\frac{a}{c}\right) = i^{\l}\chi(\conj{a})\e_{a}^{-2\l}\legendre{c}{a}c^{1-2s}\L_{f}\left(1-s,-\frac{\conj{a}}{c}\right).
    \end{equation}
  \end{proposition}
  \begin{proof}
    Let $\g = \begin{psmallmatrix} a & b \\ c & d \end{psmallmatrix} \in \G_{0}(4N)$. Set $z = -\frac{d}{c}+\frac{i}{tc}$ for $t > 0$ so that $\g z = \frac{a}{c}-\frac{t}{ic}$. Then
    \begin{equation}\label{equ:modular_transform_1}
      f\left(\frac{a}{c}-\frac{t}{ic}\right) = \sum_{n \ge 1}f_{n}n^{\frac{\l-1}{2}}e^{-2\pi n\frac{t}{c}}e\left(n\frac{a}{c}\right).
    \end{equation}
    On the other hand, by modularity
    \[
      f\left(\frac{a}{c}-\frac{t}{ic}\right) = i^{\l}\chi(d)\e_{d}^{-2\l}\legendre{c}{d}\sum_{n \ge 1}f_{n}n^{\frac{\l-1}{2}}e^{-2\pi n\frac{1}{tc}}e\left(-n\frac{d}{c}\right)t^{-\l}.
    \]
    As $d \equiv \conj{a} \tmod{c}$, $d \equiv a \tmod{4}$, and the modified Jacobi symbol is quadratic, we can express the previous identity as
    \begin{equation}\label{equ:modular_transform_2}
      f\left(\frac{a}{c}-\frac{t}{ic}\right) = i^{\l}\chi(\conj{a})\e_{a}^{-2\l}\legendre{c}{a}\sum_{n \ge 1}f_{n}n^{\frac{\l-1}{2}}e^{-2\pi n\frac{1}{tc}}e\left(-n\frac{\conj{a}}{c}\right)t^{-\l}.
    \end{equation}
    Taking the Mellin transform of \cref{equ:modular_transform_1} at $s+\frac{\l-1}{2}$ in the variable $t$ we obtain
    \begin{equation}\label{equ:additive_twist_mellin_1}
      c^{s+\frac{\l-1}{2}}\L_{f}\left(s,\frac{a}{c}\right) = \int_{0}^{\infty}f\left(\frac{a}{c}-\frac{t}{ic}\right)t^{s+\frac{\l-1}{2}}\,dt.
    \end{equation}
    Similarly, taking the Mellin transform of \cref{equ:modular_transform_2} yields
    \begin{equation}\label{equ:additive_twist_mellin_2}
      i^{\l}\chi(\conj{a})\e_{a}^{-2\l}\legendre{c}{a}c^{(1-s)+\frac{\l-1}{2}}\L_{f}\left(1-s,-\frac{\conj{a}}{c}\right) = \int_{0}^{\infty}f\left(\frac{a}{c}-\frac{t}{ic}\right)t^{s+\frac{\l-1}{2}}\,dt.
    \end{equation}
    Equating these two transforms gives to the desired functional equation. The analytic continuation follows from the fact that the integrals in \cref{equ:additive_twist_mellin_1,equ:additive_twist_mellin_2} are analytic via the rapid decay of $f$ and the functional equation just established.
  \end{proof}

\section{An Extension of Hecke's Modularity Argument}
  Let $f(z)$, $\g(s)$, and $\Lambda_f\left(s,\frac{a}{c}\right)$ as in \cref{thm:ConverseTheorem}. The aim of this appendix is to provide a variant of Hecke's modularity relation when $\g(s)$ is given explicitely.

  \begin{proposition}\label{prop:Hecke_modularity_variant}
    Let $f(z)$, $\g(s)$, $\w$, and $\Lambda_f\left(s,\frac{a}{c}\right)$ be as in \cref{thm:ConverseTheorem}. Moreover, suppose that $\g(s) = CH^{s}\G_{\C}\left(s+\frac{\nu-1}{2}\right)$ as in \cref{prop:specify_gamma_factor_if_integral_vanishes}. Then
    \[
      f\bigg\vert\begin{pmatrix}aH^2 & -1\\ cH^2 & 0\end{pmatrix}=\w i^{-\nu}\chi(\overline{a})\epsilon_a^{-2\nu}\legendre{c}{a}f\bigg\vert\begin{pmatrix}
        1 & -\overline{a}\\ 0 & c \end{pmatrix}.
    \]
  \end{proposition}
  \begin{proof}
    Let $s = \s+it$ with $ \s > \frac{3}{2}$. Then $L_{f}(s,\frac{a}{c}) \ll 1$ and by Stirling's formula $\g(s) \ll (1+t)^{-1}$ uniformly as $t \to \infty$. Therefore the Mellin inverse of $\L_{f}\left(s,\frac{a}{c}\right)$ exists. The inverse Mellin transform of $\L_{f}\left(s,\frac{a}{c}\right)$ at $\frac{1}{cHy}$ in the variable $s$ is
    \[
      \frac{1}{2\pi i}\int_{\Re(s)=\s}\left(\frac{1}{cHy}\right)^{-s}\Lambda_f\left(s,\frac{a}{c}\right)\,ds.
    \]
   Opening up $\Lambda_f\left(s,\frac{a}{c}\right) = \g(s)L_{f}\left(s,\frac{a}{c}\right)$ and interchanging with the integral gives
    \begin{equation}\label{equ:Hecke_int_1}
      2C\left(\frac{1}{cH^2y}\right)^{\frac{\nu-1}{2}}\sum_{n\geq 1}f_nn^{\frac{\nu-1}{2}}e\left(n\frac{a}{c}\right)\frac{1}{2\pi i}\int_{\Re(s)=\s}\left(\frac{2\pi n}{cH^2y}\right)^{-\left(s+\frac{\nu-1}{2}\right)}\Gamma\left(s+\frac{\nu-1}{2}\right)\,ds.
    \end{equation}
    The integral in \cref{equ:Hecke_int_1} is the inverse Mellin transform of an exponential function:
    \begin{equation}\label{equ:Hecke_int_2}
      \frac{1}{2\pi i}\int_{\Re(s)=\s}\left(\frac{2\pi n}{cH^2y}\right)^{-\left(s+\frac{\nu-1}{2}\right)}\Gamma\left(s+\frac{\nu-1}{2}\right)\,ds = e^{-\frac{2\pi n}{cH^2y}}.
    \end{equation}
    Substiuting \cref{equ:Hecke_int_2} into Expression (\ref{equ:Hecke_int_1}) and rewriting the resulting Fourier series, we obtain
    \begin{equation}\label{equ:Hecke_modularity_variant_side_1}
      2C \left(\frac{1}{cH^2y}\right)^{\frac{\nu-1}{2}}f\left(\frac{\frac{-1}{H^2yi}+a}{c}\right).
    \end{equation}
    Now let $\s < \frac{1}{2}$. The same esimates for $\L_{f}\left(1-s,\frac{a}{c}\right)$ hold as for $\L_{f}\left(s,\frac{a}{c}\right)$ in the region $\s > \frac{3}{2}$, so the inverse Mellin transform of the right-hand side in \cref{equ:functional_equation_assumption} exists. Taking the inverse Mellin transform, and noting that $2\nu \equiv \d \tmod{4}$ we obtain
    \[
      \frac{1}{2\pi i}\int_{\Re(s)=\s}\left(\frac{1}{cHy}\right)^{-s}\w\chi(\overline{a})\epsilon_a^{-2\nu}\legendre{c}{a}c^{1-2s}\Lambda_f\left(1-s,-\frac{\overline{a}}{c}\right)\,ds.
    \]
    Opening up $\L_{f}\left(1-s,\frac{a}{c}\right) = \g(1-s)L_{f}\left(1-s,\frac{a}{c}\right)$ and gathering terms with a power of $1-s$ together gives
    \[
      2CHy\w\chi(\overline{a})\epsilon_a^{-2\nu}\legendre{c}{a}\frac{1}{2\pi i}\int_{\Re(s)=\s}\left(\frac{c}{y}\right)^{1-s}(2\pi)^{1-s-\frac{\nu-1}{2}}\Gamma\left(1-s+\frac{\nu-1}{2}\right)\sum_{n\geq 1}\frac{f_ne\left(-n\frac{\overline{a}}{c}\right)}{n^{1-s}}\,ds.
    \]
    Interchanging the sum and integral, and multiplying and dividing by $\left(\frac{y}{c}\right)^{\frac{\nu-1}{2}}$, we arrive at 
    \begin{equation}\label{equ:Hecke_int_3}
      \begin{aligned}
        2CHy\w\chi(\overline{a})\epsilon_a^{-2\nu}\legendre{c}{a}\left(\frac{y}{c}\right)^{\frac{\nu-1}{2}}&\sum_{n\geq 1}f_nn^{\frac{\nu-1}{2}}e\left(-n\frac{\overline{a}}{c}\right) \\
        &\cdot\frac{1}{2\pi i}\int_{\Re(s)=\s}\left(\frac{2\pi ny}{c}\right)^{-\left(1-s+\frac{\nu-1}{2}\right)}\Gamma\left(1-s+\frac{\nu-1}{2}\right)\,ds.
      \end{aligned}
    \end{equation}
    The integral in Expression (\ref{equ:Hecke_int_3}) is the inverse Mellin transform of an exponential function:
    \begin{equation}\label{equ:Hecke_int_4}
      \frac{1}{2\pi i}\int_{\Re(s)=\s}\left(\frac{2\pi ny}{c}\right)^{-\left(1-s+\frac{\nu-1}{2}\right)}\Gamma\left(1-s+\frac{\nu-1}{2}\right)\,ds = e^{-\frac{2\pi ny}{c}}.
    \end{equation}
    Inserting \cref{equ:Hecke_int_4} into Expression (\ref{equ:Hecke_int_3}) and rewriting the resulting Fourier series yields
    \begin{equation}\label{equ:Hecke_modularity_variant_side_2}
      2CHy\w\chi(\overline{a})\epsilon_a^{-2\nu}\legendre{c}{a}\left(\frac{y}{c}\right)^{\frac{\nu-1}{2}}f\left(\frac{iy-\overline{a}}{c}\right).
    \end{equation}
    Expressions (\ref{equ:Hecke_modularity_variant_side_1}) and (\ref{equ:Hecke_modularity_variant_side_2}) are equal since they are invese Mellin transforms of either side of \cref{equ:functional_equation_assumption}. But this equality holds for all $iy$ with $y > 0$ and hence for all $z$ by the identity theorem. As a result, we obtain
    \[
      2C \left(\frac{1}{-cH^2iz}\right)^{\frac{\nu-1}{2}}f\left(\frac{\frac{-1}{H^2z}+a}{c}\right)=2CH(-iz)\w\chi(\overline{a})\epsilon_a^{-2\nu}\legendre{c}{a}\left(\frac{-iz}{c}\right)^{\frac{\nu-1}{2}}f\left(\frac{z-\overline{a}}{c}\right).
    \]
    Moving all constants to the right-hand side gives the more compact relation
    \[
      f\left(\frac{\frac{-1}{H^2z}+a}{c}\right)
      =
      \w(-iHz)^{\nu}\chi(\overline{a})\epsilon_a^{-2\nu}\legendre{c}{a}f\left(\frac{z-\overline{a}}{c}\right),
    \]
    which is equivalent to the statement of the proposition in terms of slash notation.
  \end{proof}

\bibliographystyle{alpha}
\bibliography{reference}

\begin{thebibliography}{HLN21}

\bibitem[BFL22]{booker2022extension}
Andrew~R Booker, Michael Farmer, and Min Lee.
\newblock An extension of venkatesh's converse theorem to the selberg class.
\newblock {\em arXiv preprint arXiv:2207.00451}, 2022.

\bibitem[CPS94]{cogdell1994converse}
Jim~W Cogdell and Ilya~I Piatetski-Shapiro.
\newblock Converse theorems for {$\GL_{n}$}.
\newblock {\em Publications Math{\'e}matiques de l'IH{\'E}S}, 79:157--214,
  1994.

\bibitem[Ham21]{hamburger1921riemannsche}
Hans Hamburger.
\newblock {\"U}ber die riemannsche funktionalgleichung der {$\z$}-funktion
  (erste mitteilung).
\newblock {\em Mathematische Zeitschrift}, 10:240--254, 1921.

\bibitem[Hec36]{hecke1936bestimmung}
Erich Hecke.
\newblock {\"U}ber die bestimmung dirichletscher reihen durch ihre
  funktionalgleichung.
\newblock {\em Mathematische Annalen}, 112(1):664--699, 1936.

\bibitem[HLN21]{hoffstein2021first}
Jeff Hoffstein, Min Lee, and Maria Nastasescu.
\newblock First moments of rankin--selberg convolutions of automorphic forms on
  {$\GL(2)$}.
\newblock {\em Research in Number Theory}, 7:1--44, 2021.

\bibitem[KRS03]{kim2003functoriality}
Henry~H. Kim, Dinakar Ramakrishnan, and Peter Sarnak.
\newblock Functoriality for the exterior square of {$\GL_{4}$} and the
  symmetric fourth of {$\GL_{2}$}.
\newblock {\em Journal of the American Mathematical Society}, 16(1):139--183,
  2003.

\bibitem[Shi73]{shimura1973modular}
Goro Shimura.
\newblock On modular forms of half integral weight.
\newblock {\em Annals of Mathematics}, 97(3):440--481, 1973.

\bibitem[Ven02]{venkatesh2002limiting}
Akshay Venkatesh.
\newblock {\em Limiting forms of the trace formula}.
\newblock Princeton University, 2002.

\bibitem[Wal85]{Waldspurger1985}
J.-L. Waldspurger.
\newblock Sur les valeurs de certaines fonctions $l$ automorphes en leur centre
  de symétrie.
\newblock {\em Compositio Mathematica}, 54(2):173--242, 1985.

\bibitem[Wei67]{weil1967bestimmung}
Andr{\'e} Weil.
\newblock {\"U}ber die bestimmung dirichletscher reihen durch
  funktionalgleichungen.
\newblock {\em Mathematische Annalen}, 168:149--156, 1967.

\bibitem[ZJ07]{zwillinger2007table}
Daniel Zwillinger and Alan Jeffrey.
\newblock {\em Table of integrals, series, and products}.
\newblock Elsevier, 2007.

\end{thebibliography}

\end{document}